\documentclass[12pt]{amsart}
\usepackage[applemac]{inputenc}

\usepackage[colorlinks=true,linkcolor=magenta,citecolor=blue,urlcolor=blue,hypertexnames=false,linktocpage]{hyperref}
\usepackage{bookmark}
\usepackage{amsthm,thmtools,amssymb,amsmath,amscd,amsfonts}
\usepackage{etoolbox}
\apptocmd{\sloppy}{\hbadness 10000\relax}{}{}

\usepackage{fancyhdr}
\usepackage{esint}
\usepackage{enumerate}

\usepackage{pictexwd,dcpic}

\declaretheorem[name=Theorem,numberwithin=section]{thm}
\declaretheorem[name=Remark,style=remark,sibling=thm]{rem}
\declaretheorem[name=Lemma,sibling=thm]{lemma}

\declaretheorem[name=Definition,style=definition,sibling=thm]{defn}
\declaretheorem[name=Corollary,sibling=thm]{cor}
\declaretheorem[name=Assumption,style=definition,sibling=thm]{assum}

\declaretheorem[name=Theorem,numbered=no]{theorem}
\declaretheorem[style=remark,name=Remark,numbered=no]{remark}

\numberwithin{equation}{section}

\usepackage{cleveref}
\crefname{lemma}{Lemma}{Lemmata}
\crefname{prop}{Proposition}{Propositions}
\crefname{thm}{Theorem}{Theorems}
\crefname{cor}{Corollary}{Corollaries}
\crefname{defn}{Definition}{Definitions}
\crefname{example}{Example}{Examples}
\crefname{rem}{Remark}{Remarks}
\crefname{assum}{Assumption}{Assumptions}
\crefname{nota}{Notation}{Notation}

\usepackage{autonum}

\newcommand{\ti}{\tilde}
\newcommand{\wt}{\widetilde}

\newcommand{\cn}{\colon}
\newcommand{\sub}{\subset}
\newcommand{\ov}{\overline}

\newcommand{\bbN}{\mathbb{N}}

\newcommand{\bbR}{\mathbb{R}}

\newcommand{\bbS}{\mathbb{S}}
\newcommand{\bbH}{\mathbb{H}}

\newcommand{\bbE}{\mathbb{E}}
\newcommand{\8}{\infty}

\newcommand{\al}{\alpha}

\newcommand{\g}{\gamma}
\newcommand{\de}{\delta}
\newcommand{\e}{\varepsilon}
\newcommand{\ka}{\kappa}
\newcommand{\la}{\lambda}

\newcommand{\s}{\sigma}
\newcommand{\Si}{\Sigma}
\newcommand{\p}{\varphi}

\newcommand{\vt}{\vartheta}
\newcommand{\Om}{\Omega}

\newcommand{\G}{\Gamma}

\newcommand{\La}{\Lambda}

\newcommand{\cL}{\mathcal{L}}

\newcommand{\cW}{\mathcal{W}}

\newcommand{\del}{\partial}
\newcommand{\n}{\nabla}

\newcommand{\fa}{\forall}

\newcommand{\ip}[2]{\left\langle #1,#2 \right\rangle}
\newcommand{\fr}[2]{\frac{#1}{#2}}
\newcommand{\x}{\times}

\DeclareMathOperator{\Rm}{Rm}

\DeclareMathOperator{\sgn}{sgn}

\newcommand{\pf}[1]{\begin{proof}#1 \end{proof}}
\newcommand{\eq}[1]{\begin{equation}\begin{alignedat}{2} #1 \end{alignedat}\end{equation}}

\newcommand{\br}[1]{\left(#1\right)}
\newcommand{\abs}[1]{\lvert #1\rvert}
\newcommand{\enum}[1]{\begin{enumerate}[(i)] #1 \end{enumerate}}

\newcommand{\Ra}{\Rightarrow}
\newcommand{\ra}{\rightarrow}
\newcommand{\hra}{\hookrightarrow}

\newcommand{\mt}{\mapsto}

\newcommand{\hp}{\hphantom}

\begin{document}

\title{Parabolic approaches to curvature equations}

\author[P. Bryan, M. N. Ivaki, J. Scheuer]{Paul Bryan, Mohammad N. Ivaki, Julian Scheuer}

\date{\today}
\dedicatory{}
\subjclass[2010]{}
\begin{abstract}
We employ curvature flows without global terms to seek strictly convex, spacelike solutions of a broad class of elliptic prescribed curvature equations in the simply connected Riemannian spaceforms and the Lorentzian de Sitter space, where the prescribed function may depend on the position and the normal vector. In particular, in the Euclidean space we solve a class of prescribed curvature measure problems, intermediate $L_p$-Aleksandrov and dual Minkowski problems as well as their counterparts, namely the $L_{p}$-Christoffel-Minkowski type problems. In some cases we do not impose any condition on the anisotropy except positivity, and in the remaining cases our condition resembles the constant rank theorem/convexity principle due to Caffarelli-Guan-Ma (Commun. Pure Appl.
Math. 60 (2007), 1769--1791). Our approach does not rely on monotone entropy functionals and it is suitable to treat curvature problems that do not possess variational structures.
\end{abstract}

\maketitle
\tableofcontents
\section{Introduction}
Let $n\in \bbN$ and $N$ be either the Euclidean space $\bbE^{n+1}$, the sphere $\bbS^{n+1}$, the hyperbolic space $\bbH^{n+1}$ with respective sectional curvature $0,1,-1$ or the $(n+1)$-dimensional Lorentzian de Sitter space $\bbS^{n,1}$ with sectional curvature $1$. It may also be possible to treat Lorentzian spaceforms of non-positive sectional curvature with similar methods but we do not address that question here since some modifications are needed.
In this paper, we solve equations of prescribed curvature in open subsets $\Om\sub N$, i.e., we seek strictly convex, spacelike hypersurfaces $\Si\sub \Om$ which satisfy
 \eq{\label{Intro}F(\ka_{1},\dots,\ka_{n})=f(x,\nu).}
Here $F$ is a strictly monotone symmetric function of the principal curvatures $\ka_{i}$ of $\Si$. Further conditions on $F$ will be imposed later. Moreover, we assume
\eq{f\in C^{\8}(\bar\Om\x\ti N),}
where $\ti N$ denotes the {\it{dual manifold}} of $N$, i.e.,
\eq{\ti \bbS^{n+1}=\bbS^{n+1},\quad \ti \bbH^{n+1}=\bbS^{n,1},\quad \ti \bbS^{n,1}=\bbH^{n+1}.}
The normal $\nu$ to a hypersurface in $N$ is then understood to be a point in $\ti N$. Here we view all these spaces as hyperquadrics in $\bbR^{n+2}_{\mu}$, where $\mu=\pm 1$ and for $\mu=1$ we mean the Euclidean space and for $\mu=-1$ the Minkowski space of dimension $n+2$ respectively.

In case $N=\bbE^{n+1}$, the normal vector to a hypersurface naturally lives in the unit $n$-sphere and hence in this case we let
\eq{f\in C^{\8}(\bar\Om\x \bbS^{n}),\quad \ti \bbE^{n+1}=\bbS^{n}.}

Equation \eqref{Intro} contains an interesting subclass, the so-called {\it{$L_{p}$-Minkowski problems}} and more general variants; see \Cref{subsec:CMP} for a detailed description. Recently an interesting branch of research has been developed which attempts to solve these equations in the smooth category with the help of curvature flows; e.g., \cite{Andrews:/1998,Andrews:09/2000,BIS3,ChouWang:11/2000,Ivaki:02/2019,LiShengWang:/2020}. To our knowledge, all of those treatments relied on the existence of entropy functionals: the curvature flows were set up as either purely expanding or contracting; therefore, the renormalization contained a \emph{global} term which is manageable whenever a monotone entropy is available. This limits the class of curvature problems that could be treated with a parabolic approach. On the other hand, employing curvature flows without global terms to show the existence of strictly convex solutions to equation \eqref{Intro} in Euclidean and non-Euclidean ambient spaces has been investigated only for a fairly restricted class of curvature functions; see e.g., \cite{Gerhardt:/1996,Gerhardt:01/1996,Gerhardt:/2000,Gerhardt:09/2000,Gerhardt:/2003}.

The aim of this paper is twofold. First, by imposing some mild structural assumptions on $f$, we prove new existence results of strictly convex solutions  to \eqref{Intro} for a large class of curvature functions $F$. These mild conditions are that either $f$ satisfies some convexity assumption with respect to $x$, or with regards to the $L_{p}$-Minkowski type problems that $f$ depends on $x$ only through the support function $s$,
\eq{f=f(s,\nu).}
Second, we follow the recent developments in the theory of smooth Alexandrov-Fenchel type inequalities, where on several occasions it has proven useful to replace the global terms in the classical flows by local terms, cf., \cite{GuanLi:/2015,GuanLiWang:/2019,ScheuerWangXia:11/2018,ScheuerXia:11/2019}. The novel feature here is that this approach does not rely on monotone energy functionals hence allowing us to solve a larger class of curvature problems that do not enjoy variational structures. This technique provides a new flow approach to further important problems arising in convex geometry such as the $L_{p}$-Christoffel-Minkowski problems for a certain range of $p$.

One may treat \eqref{Intro} by considering a parabolic flow of spacelike hypersurfaces such as
\eq{\label{Contr}\dot{x}=\s(f-F)\nu,}
where $x$ is the time-dependent embedding vector, $\nu$ is the future directed (outward) unit normal (we will only consider cases where such a choice is possible) and where
\eq{\s:=\ip{\nu}{\nu}}
is the signature of the $(n+1)$-dimensional respective ambient space.
If $F>0$, a flow of the form \eqref{Contr} can be considered as a {\it{contracting type}}, because it is a lower order perturbation of contracting flows such as the mean curvature flow. In contrast, one could also consider
\eq{\label{Exp}\dot{x}=\s\br{\fr{1}{F}-\fr{1}{f}}\nu,}
which may be considered as a flow of {\it{expanding type}}.

Depending on the type of elliptic problem \eqref{Intro}, it is favorable to use either \eqref{Contr} or \eqref{Exp}. As we will use both versions, we will derive the evolution equations for the general form
\eq{\label{x}\dot{x}=\s\br{\Phi(f)-\Phi(F)}\nu,}
where $\Phi\cn (0,\8)\ra \bbR$
is any nowhere vanishing function with $\Phi'>0$. For the contracting and expanding types described above, $\Phi(y)=y$ and $\Phi(y)=-y^{-1}$ respectively. Whenever the convergence of \eqref{x} to a steady state is shown, we have a solution to \eqref{Intro}. We prove results of the type:

\begin{theorem}
Let $F$ be a degree one homogeneous curvature function and $f = f(x, s, \nu)$ be a positive smooth function of the position, support function and unit normal. Then under certain structural (concavity/convexity) assumptions, the flow \eqref{x} exists for all time and smoothly converges to a solution of the prescribed curvature problem \eqref{Intro}.
\end{theorem}

The results include but are not limited to the Minkowski problem, $L_p$-Minkowski problem, dual Minkowski problem, the $L_p$-Christoffel-Minkowski problem and generalizations thereof. Full details of the assumptions and precise results are described below in \Cref{sec:results}. Some further remarks are in order describing the methods.

When studying problems of prescribed curvature and convex hypersurface flows, it is common to reduce the problem to a problem in terms of the support function (see e.g., \cite{Urbas:/1991}) in which the curvature function $F$ is rewritten in terms of the principal radii $r_i = \kappa_i^{-1}$. Such a reformulation naturally leads to the condition of inverse concavity (see \Cref{F1} and \Cref{F2} below). Compare with prior work on prescribed curvature equations (e.g., \cite{Gerhardt:02/1997,Gerhardt:/2000}) where the speed is assumed log-concave leading to upper curvature bounds. With the additional assumption that $F$ vanishes on the boundary $\partial\G_{+}$ of the convex cone of principal curvatures, lower curvature bounds are obtained. In our situation, we obtain lower curvature bounds using inverse concavity and obtain upper bounds by requiring the inverse function $F_{\ast}$ vanishes on the boundary of the convex cone as in \Cref{F1} and \Cref{Flow-Main} below. When formulated in terms of the support function, the inverse concavity condition allows us to then apply Krylov-Safonov estimates for higher regularity in Euclidean space whilst in other cases we employ duality (see the proof of \Cref{Flow-Main}).

One complication here is that in order to obtain lower bounds, the prescribed function $f$ is required to satisfy strict convexity properties in the position $x$, similar to weak convexity assumptions of \cite{CaffarelliGuanMa:12/2007,GuanRenWang:/2013}. Although we require the stronger condition of strict convexity, we do not require any convexity assumption on the $x$-dependence of $f$ via the support function. We also allow arbitrary anisotropy in $\nu$. This freedom allows us to broaden the scope to include problems of $L_p$-Minkowski type in \Cref{subsec:CMP} and \Cref{subsec:christoffel-minkowski}.

We also investigate dropping the assumption that the inverse function $F_{\ast}$ vanishes on $\partial \G_{+}$. This allows us to treat a broader class of flows such as those with speed $\left(\sigma_{\ell}/\sigma_k\right)^{1/(\ell-k)}$, where $\sigma_{k}$ is the $k$-th elementary symmetric polynomial of the principal curvatures, but the upper curvature estimates no longer follow automatically. Also, in view of the counterexamples of \cite[Thm.~1.2]{GuanRenWang:/2013} and \Cref{rem:firey}, some convexity-type restrictions are required on the prescribed function $f$ and we describe a natural such condition in \Cref{subsec:christoffel-minkowski} as arising in the $L_{p}$-Christoffel-Minkowski problem.

\section{Definitions and notation}

For brevity we rewrite \eqref{x} as
\eq{\label{Flow}\dot{x}=\s(\phi-\Phi)\nu,}
where
\eq{\phi=\Phi(f),\quad \Phi=\Phi(F).}

\begin{defn}\label{N} We fix the following conventions.
\enum{
\item
$n\in \bbN$ and $N$ denotes either $\bbE^{n+1}$, $\bbS^{n+1}$, $\bbH^{n+1}$ or $\bbS^{n,1}$, while the dual manifolds $\ti N$ are described in the introduction. The sectional curvature of $N$ is denoted by $K_{N}$. The signature of these spaces is denoted by $\s$.
\item The spaces described in item (i) with $K_{N}\neq 0$ all naturally arise as hyperquadrics in an $(n+2)$-dimensional Euclidean or Minkowski space respectively, whose signature is $\mu\in \{\pm 1\}$ throughout this paper.
\item We write $\ti g, \ti D$ for the metric and the corresponding Levi-Civita connection of $\ti N$, while $\bar{D}$ denotes the Levi-Civita connection of the metric $\bar g$ of $N$.
\item For $\Om\sub N$, on $\Om \x \ti N$ we use the product metric and its Levi-Civita connection.
\item The induced metric on $\Si$ is denoted by $g$ and its Levi-Civita connection is $\n$.
For its (and every other metric's) Riemann tensor we use the conventions
\eq{\Rm(X,Y)Z=\n_{X}\n_{Y}Z-\n_{Y}\n_{X}Z-\n_{[X,Y]}Z}
and
\eq{\Rm(X,Y,Z,W)=g(\Rm(X,Y)Z,W).}
\item We occasionally use a local coordinate frame $(\del_{i})_{1\leq i\leq n}$ to calculate tensor expressions. We denote
\eq{R_{ijkl}=\Rm(\del_{i},\del_{j},\del_{k},\del_{l}).}
 For covariant derivatives along $\Si$, in order to facilitate the notation we use semi-colons to denote covariant derivatives, e.g., the components of the second derivative $\n^{2}T$ of a tensor are denoted by
\eq{T_{;ij}:=\n_{\del_{j}}\n_{\del_{i}}T-\n_{\n_{\del_{j}\del_{i}}}T.}
\item We use the following Gaussian formula:
\eq{\bar D_{X}Y=\n_{X}Y-\sigma h(X,Y)\nu,}
where $\nu$ will consistently be chosen such that
\eq{\s\bar g(\nu,\del_{r})>0.}
Here $r$ is the radial coordinate in geodesic normal coordinates, see \Cref{Om}.
The Weingarten operator $A$ is defined by
\eq{g(A(X),Y)=h(X,Y)} and its eigenvalues, the principal curvatures of the hypersurface, are ordered as:
\eq{\ka_{1}\leq\dots\leq\ka_{n}.}
\item We write
\eq{\bbS^{n+1}_{+}&=\{x\in \bbR^{n+2}\cn \abs{x}_{+}=1, x^{n+2}>0\},\\
 \bbS^{n,1}_{+}&=\{x\in \bbR^{n+2}\cn \abs{x}_{-}=1, x^{n+2}>0 \},}
where $\abs{x}_{+}^{2}=\ip{x}{x}_{+}$ and $\abs{x}_{-}^{2}=\ip{x}{x}_{-}$ are the Euclidean resp, Minkowskian norms and dot products of signature $(1,\dots,1,\mu)$.
}
\end{defn}

Also recall that for a spacelike hypersurface of $N$ the Gauss equation is given by
\eq{R_{ijkl}=\s(h_{il}h_{jk}-h_{ik}h_{jl})+\ov\Rm(x_{;i},x_{;j},x_{;k},x_{;l}),}
cf., \cite[Thm.~8.4]{Lee:/1997} for the Riemannian case. The Lorentzian case differs from this only in $\bar g(\nu,\nu)$.

\begin{assum}\label{Om}
Let $\Om\Subset N$ be a {\it{strict annular region}}, i.e., the subspace $\Om$ is isometric to
\eq{\label{ann}\Om=(a,b)\x\bbS^{n},\quad \bar g=\s dr^{2}+\vt^{2}(r)\hat g,\quad a<b,}
where $\hat g$ denotes the round metric on $\bbS^{n}$ and $\vt\in C^{\8}([a,b])$ is positive.
In the sequel, for all such strict annular regions we will assume that $\vt'>0$. In the cases $N=\bbS^{n+1}$ and $N=\bbS^{n,1}$, this restricts us to $\bbS_{+}^{n+1}$ and $\bbS^{n,1}_{+}$ respectively.
\end{assum}

For a graphical hypersurface in $\Om$,
\eq{\Si:=\{(r(y),y)\cn y\in \bbS^{n}\},}
a support function $s$ can be defined by
\eq{s=\s\bar g(\vt\del_{r},{\nu}).}
We always choose the normal of a graphical hypersurface to be future directed. We have
\eq{\label{support}s=\fr{\vt}{v},}
where
\eq{\label{v}v^{2}=1+\s \vt^{-2}\hat g^{ij}r_{;i}r_{;j}.}
This is easily seen from the coordinate expression of $\nu$ in the given coordinate system.
With this choice, the support function, as well as the principal curvatures of the slices $\{r\}\x \bbS^{n}$ are always positive and in particular given by
\eq{\bar\ka_{i}=\fr{\vt'(r)}{\vt(r)}.}

We artificially introduce the variable $s$ into the domain of $f$, in order to distinguish between properties that $f$ must satisfy with respect to $x$ but not with respect to $s$. We define the positive cone as
\eq{\G_{+}=\{\ka\in\bbR^{n}\cn \ka_{i}>0\quad\fa 1\leq i\leq n \}.}
In the sequel $\s_{k}$ denotes the $k$-th elementary symmetric polynomial
\eq{\s_{k}(\ka)=\sum_{1\leq i_{1}<\dots<i_{k}\leq n}\ka_{i_{1}}\dots \ka_{i_{k}}.}

\begin{defn}
Let $F\in C^{\8}(\G_{+})$, $\Om\sub N$ open and
\eq{f\in C^{\8}(\bbR_+\x \bar{\Om}\x\ti N). }
We say that a closed, strictly convex hypersurface $M\sub \Om$ is a lower barrier for the pair $(F, f)$, if
\eq{\label{lower barrier}\s( f(s,x,\nu)-F(\ka_{i}))_{|M}\geq 0}
and an upper barrier, if
\eq{\label{upper barrier}\s( f(s,x,\nu)-F(\ka_{i}))_{|M}\leq 0.}
\end{defn}

\begin{defn}
A $1$-homogeneous positive function $F\in C^{\8}(\G_{+})$ is inverse concave if
		\eq{F_{\ast}(\ka_{i})=\fr{1}{F(\ka_{i}^{-1})}~\hbox{is concave}.}
\end{defn}

\section{Results}\label{sec:results}

\subsection{Prescribed curvature equations}

\begin{assum}\label{F1}
Suppose
\enum{
\item $F\in C^{\8}(\G_{+})$ is a positive, strictly monotone, $1$-homogeneous function normalized to $F(1,\dots,1)=n,$
\item If $N=\bbE^{n+1}$ or $\bbH^{n+1}$, $F$ is inverse concave and $F_{\ast|\del\G_{+}}=0,$
\item if $N=\bbS^{n+1}$, $F$ is concave and inverse concave and $F_{\ast|\del\G_{+}}=0$
\item if $N=\bbS^{n,1}_{+}$, $F$ is convex and uniformly monotone up to $\del\G_+.$
}
\end{assum}

\begin{remark}
If $F$ is $1$-homogeneous and convex, then it is inverse concave. Moreover, since due to convexity $F\geq H$, the inverse of $F$ vanishes on $\del\G_+$. Later in this paper, we will investigate the effect of dropping the restriction $F_{\ast|\del\G_{+}}=0$; see \Cref{F2}.
\end{remark}

The first main results of this paper is stated in the next theorem.

\begin{thm}\label{Flow-Main}
Let $\Om$ be a strict annular region of $N$ as described in \Cref{Om}.
Suppose $F$ satisfies \Cref{F1} and $f\in C^{\8}(\bbR_+\x \bar\Om\x\ti N)$ is a positive function. Suppose $M_{0}=\{a\}\x\bbS^{n}$  and $\ti M_{0}=\{b\}\x\bbS^{n}$ is a lower resp. an upper barrier for the pair $(F,f).$
 Suppose $\phi=-f^{-1}$ satisfies
\eq{\label{Flow-Main-1}\bar D^{2}_{xx}\phi+K_{N}\phi\bar g< 0.}
Then there exists a unique smooth time-dependent family of embeddings
\eq{x\cn [0,\8)\x \bbS^{n}\ra \bar\Om}
satisfying the curvature flow equation
\eq{\label{Flow-Main-2}\dot{x}&=\s\br{\phi(s,x,\nu)+\fr{1}{F}}\nu,}
where in the Riemannian case we start the flow from $M_{0}$, while in the Lorentzian case we start it from $\ti M_{0}$.
The embeddings $x(t,\cdot)$ converge to a strictly convex solution of the generalized Minkowski problem
\eq{\label{GMP}F(\ka)=f(s,x,\nu).}
\end{thm}

\begin{rem}~
\enum{
\item In the case $N=\bbE^{n+1}$, we have $s=\langle x,\nu\rangle$, so we may think of $f(s,x,\nu)$ purely as a function of $(x,\nu)$ as long as $s>0.$ The point of introducing $s$ as an independent variable is that the convexity assumption \eqref{Flow-Main-1} does not apply to the $x$-dependence of $s$, cf., \Cref{cor1}. In other ambient spaces, we can relax the convexity assumption as well, by allowing $f$ to depend on $s$ without any restriction. Hence our (strict) convexity principle crucially differs from the convexity assumptions with respect to $x$ in \cite{CaffarelliGuanMa:12/2007,ChenLiWang:04/2018,GuanRenWang:/2013}. Note that compared to \cite[Thm.~1.2]{CaffarelliGuanMa:12/2007}, our extra assumption $F_{\ast|\del\G_{+}}=0$ is merely used for deriving $C^2$ estimates, essential in deducing \textit{the long time behavior} of the flow.

\item For $F=\sigma_k^{1/k}$ and $N=\bbE^{n+1},$ let us compare \Cref{Flow-Main} with a similar existence result obtained by Guan-Ren-Wang in \cite[Thm.~1.5]{GuanRenWang:/2013}, where the authors prove the following statement. Suppose $f\in C^{2}(\bbE^{n+1}\times \bbS^n)$ is a positive function and \eqref{upper barrier} holds for $M=B_r(0)$ for some $r>0$. Moreover, assume $f^{-\fr 1k}$ is \textit{weakly} convex with respect to $x$; i.e.,
\eq{\bar D^2_{xx}f^{-\fr 1k}\geq 0.}
Then there exists a strictly convex $C^{3,\alpha}$-solution inside $B_r(0)$ to the equation
\eq{\s_{k}=f(x,\nu).}
In \eqref{Flow-Main-1}, we have a strict inequality to ensure the strict convexity of solutions and that we can obtain curvature bounds along the flow. In contrast, in \cite{GuanRenWang:/2013}, the strict convexity of the solution follows from a sophisticated constant rank theorem. Also note that we assume both lower and upper barrier conditions. Our lower barrier assumption is required as we allow $f\in C^{\8}(\bbR_+\times\bbE^{n+1}\times\bbS^{n})$ to become singular at $s=0$, while in \cite{GuanRenWang:/2013}, $f\in C^{\8}(\bbE^{n+1}\times\bbS^{n})$. Consequently the $L_{p}$-Minkowski problem described below is not covered by the existence theorem \cite[Thm.~1.5, case $k=n$]{GuanRenWang:/2013}: The upper barrier condition requires us to have $p>n+1$ (see \Cref{KN=1}, \Cref{thm:Lp-Mink} and \Cref{rem:Lp-Mink} below), but this forces $f$ to blow up at the origin.
}
\end{rem}

To clarify item (i) in the previous remark that the support function does not play any role in the convexity assumption of $f^{-1}$ with respect to $x$, we state two corollaries of \Cref{Flow-Main}.

\begin{cor}\label{cor1}
Let $N=\mathbb{E}^{n+1}$ and $p+k<n+1$. Suppose $\p\in C^{\infty}(\mathbb{S}^n)$ is positive and
\eq{|x|^{\frac{n+1}{k}}\p\left(\frac{x}{|x|}\right)^{-\frac{1}{k}}~\mbox{is strictly convex on}~ \mathbb{R}^{n+1}\setminus\{0\}.}
Then the prescribed curvature measure problem
\eq{\sigma_k=\frac{s^{p}}{|x|^{n+1}}\p\left(\frac{x}{|x|}\right)}
admits a strictly convex solution.
\end{cor}

\begin{proof}
We let $F = \sigma_k^{1/k}$ and
\eq{f(x, s) = \frac{s^{p/k}}{|x|^{(n+1)/k}} \p\br{\frac{x}{\abs{x}}}^{1/k}} in \Cref{Flow-Main}. The concavity requirement for $\phi$ follows immediately from the hypothesis that $1/f$ is convex in $x$. In view of \cite[Lem. 2.2.11]{Gerhardt:/2006}, $F$ satisfies \Cref{F1}. Due to the range of $k,p$ the barrier assumption is satisfied as well: the scaling properties of $\ka,s$ and $\abs{x}$ imply that $f$ scales like $\lambda^{(p-n-1)/k}$ while $F$ scales like $\lambda^{-1}$. Thus provided $p - n - 1 < -k$, for sufficiently large $\la$ the barrier assumption holds in the region
\eq{\Om=\{\la^{-1}<\abs{x}<\la\}.} Hence we can apply \Cref{Flow-Main}.
\end{proof}

Of course here we could allow non-homogeneous functions of $s$ as well as functions of the outer unit normal as long as the barrier assumption is satisfied.
For $p=1$, this theorem was first proved in \cite[Thm.~1.1]{GuanLinMa:/2009} and then under a weaker convexity assumption in \cite[Thm.~3.7]{GuanLiLi:/2012}. The case $k=n$ and $p=1$ is known as the Aleksandrov problem (Gauss curvature measure problem) which is solved completely and does not require the convexity assumption \cite{GuanLi:08/1997}.

\begin{cor}
Let $N=\mathbb{E}^{n+1}$. Suppose $f\in C^{\infty}(\mathbb{S}^n)$ is positive. Then there exists a strictly convex solution to each of the following problems.
\eq{\label{intermediate dual}|x|^{n+1-q}\sigma_k=sf(\nu)~\mbox{when}~q+k<n,}
\eq{\label{intermediate Lp Alek}|x|^{n+1}\sigma_k=s^{1-p}f(\nu)~\mbox{when}~p>k-n.}
\end{cor}

\begin{proof}
This is similar to the proof of \Cref{cor1}. Note that for equation \eqref{intermediate dual}, scaling forces $q + k < n$ which also ensures convexity of $|x|^{\frac{n+1-q}{k}}$. For equation \eqref{intermediate Lp Alek} convexity is automatic.
\end{proof}

The case $k=n$ of \eqref{intermediate dual} with $q\in \mathbb{R}$ is known as the Minkowski problem for $q$-th dual curvature measure, a.k.a. the dual Minkowski problem which in particular includes the logarithmic Minkowski problem $(q=n+1)$ and the Aleksandrov problem ($q=0$). The dual Minkowski problem was introduced in \cite{HuangLutwakYangZhang:06/2016}, see also \cite{ChenHuangZhao:07/2019,HuangJiang:10/2019,HuangZhao:07/2018,LiShengWang:/2020,Zhao:/2018}. The case $k=n$ of \eqref{intermediate Lp Alek} is known as the $L_p$-Aleksandrov problem which was introduced and studied for measures in \cite{HuangLutwakYangZhang:02/2018}. For results on the uniqueness of solutions see \cite{LutwakYangZhang:04/2018,HuangZhao:07/2018}. The curvature problems \eqref{intermediate dual} and \eqref{intermediate Lp Alek} for $1\leq k<n$ can be considered as the intermediate cases of the dual Minkowski and $L_p$-Aleksandrov problems. The former was introduced in \cite{LiShengWang:03/2020} where due to lack of variational structure no existence result was obtained. Our result seems to be the first that addresses the issue of existence and convexity.

\subsection{Generalized {\texorpdfstring{$L_{p}$-Minkowski}{Lp-Minkowski}} problems}\label{subsec:CMP}

An important class of curvature problems which includes the $L_{p}$-Minkowski problem is
\eq{\label{s-nu-flow}F=f(s,\nu).}
The smooth $L_{p}$-Minkowski problem seeks a smooth strictly convex solution to the equation
\eq{\label{Lp-Mink}\s_{n}=cs^{1-p}\p(\nu)~\hbox{for some constant}~c>0,}
where $\p\in C^{\8}(\ti N)$ and $p\in \bbR$.
This problem is a generalization of the Minkowski problem (i.e., $p=1$) which was put forward by Lutwak \cite{Lutwak:/1993} almost a century after Minkowski's original work and stems from Firey's $L_p$-linear combination of convex bodies \cite{Firey:/1962}; also see \cite{BianchiBoroczkyColesanti:02/2020,BianchiBoroczkyColesantiYang:01/2019,BoroczkyLutwakYangZhang:06/2013,ChouWang:09/2006,ChengYau:/1976,Ivaki:11/2019,LutwakOliker:/1995}. We proceed with a generalized $L_{p}$-Minkowski problem in $\bbS^{n+1}$.

\begin{cor}\label{KN=1}
Let $N=\bbS^{n+1}$ and $\p\in C^{\8}(\bbS^{n+1})$ be positive. Suppose $q>2$ and $F$ satisfies \Cref{F1}. Then there exists a strictly convex solution in $\bbS^{n+1}_{+}$ to
\eq{F=cs^{1-q}\p(\nu)~\hbox{for some positive constant}~c.}
\end{cor}
\pf{
For $s>0$ define
\eq{f=cs^{1-q}\p(\nu)=c(\bar g(\vt\del_{r},\nu))^{1-q}\p(\nu),}
where $c$ will be chosen to ensure the existence of barriers in $\bbS^{n+1}_{+}$.
We have $\vt(r)=\sin r$ and the principal curvatures of the geodesic spheres satisfy
\eq{\bar \ka_{i}=\fr{\vt'}{\vt}=\cot r.}
Let $0<b<\tfrac{\pi}{2}$ and choose
\eq{0<c<\inf_{\nu\in\bbS^{n+1}}\fr{n\cos b}{\p(\nu)\sin^{2-q}b},}
then on geodesic spheres we have
\eq{(f-F)_{|r=b}=c\p\sin^{1-q}b-n\cot b<0}
where we used $1$-homogeneity and the normalization of $F$. On the other hand there holds
\eq{\liminf_{r\ra 0}(c\p\sin^{1-q}r-n\cot r)>0}
due to $q>2$.
Defining $\Om=(a,b)\x \bbS^{n}$ with $b$ as above and $a$ sufficiently small, the assumptions of \Cref{Flow-Main} are satisfied and we obtain our solution.
}

The proof of the next theorem regarding the Euclidean case will be given in \Cref{pf:Lp-Mink}.

\begin{thm}\label{thm:Lp-Mink}
Let $N=\bbE^{n+1}$, $q>2$ and $\p\in C^{\8}(\bbS^{n})$ be positive. Suppose that $F$ satisfies \Cref{F1}. Then there exists a strictly convex solution to \eq{F=s^{1-q}\p(\nu).}
\end{thm}

\begin{rem}\label{rem:Lp-Mink}~
\enum{
\item Unlike \Cref{KN=1}, when $K_N = 0$, the lack the background curvature term means \eqref{Flow-Main-1} is not satisfied and \Cref{thm:Lp-Mink} does not follow from \Cref{Flow-Main}.

\item These last two results are formulated for $1$-homogeneous $F$; hence one has to transform the parameters to get the formulation \eqref{Lp-Mink} which then holds for $p = n(q-1) + 1 > n+1$.

\item For the case $K_{N}<0$, our approach does not seem to yield an existence result due to a bad zero order term. In \Cref{Flow-Main}, this bad zero order term was compensated using the Hessian of $\phi$ with respect to $x$, cf., \eqref{Flow-Main-1}. However in the $L_p$-Minkowski problem, such a term is not present and hence we do not obtain an existence result of this kind in the hyperbolic space.
}
\end{rem}

\subsection{Generalized \texorpdfstring{$L_{p}$-Christoffel-Minkowski problem}{Lp-Christoffel-Minkowski problems}}
\label{subsec:christoffel-minkowski}

In this subsection, we give existence results under the following structural assumptions on $F$.
\begin{assum}\label{F2} Suppose
\enum{
\item $F\in C^{\8}(\G_{+})$ is a positive, strictly monotone, $1$-homogeneous function normalized to
$F(1,\dots,1)=n,$
\item $F$ is concave and inverse concave.
}
\end{assum}
Note that compared to \Cref{F1}-(ii) we dropped
\eq{\label{zero boundary}F_{\ast|\del\G_{+}}=0,}
but added the concavity assumption on $F$.

For the previous results, \eqref{zero boundary} is crucial in obtaining upper curvature bounds from the lower curvature bounds and the upper bound on $F$ (see \Cref{pf:Flow-Main} below).
In view of the counterexamples of \cite[Thm.~1.2]{GuanRenWang:/2013} and \Cref{rem:firey}, elliptic/parabolic approaches based on a priori estimates \textit{do not lead} to a solution when $F_{\ast|\del\G_{+}}\neq0$, unless further assumptions on $\varphi$ are imposed.

A class of curvature functions that does not satisfy \eqref{zero boundary} is
\eq{F=\br{\fr{\s_{n}}{\s_{n-k}}}^{\fr 1k}~\hbox{for}~1\leq k<n.}
The corresponding problem is the {\it{$L_{p}$-Christoffel-Minkowski problem}}
\eq{\label{CMP}\p s^{1-p}\s_{k}(\ka_{i}^{-1})=c.}
This problem has been studied in \cite{Firey:/1968,Firey:12/1970,GuanMa:03/2003,GuanXia:04/2018,HuMaShen:10/2004}. In the following remark, we will discuss a natural $C^2$ condition that we may impose on $\varphi$ in order to find a strictly convex solution of \eqref{CMP}.

\begin{rem}\label{rem:firey}
Suppose $1< k<n$. The regular Christoffel-Minkowski problem asks for necessary and sufficient conditions on a positive function $\psi\in C^{\infty}(\bbS^{n})$ in order for $\psi$ to be the $\sigma_k(\kappa_i^{-1})$ of a strictly convex body, i.e., this is equation \eqref{CMP} with $p=1$ and $\p=1/\psi$. Regarding this case, we will now recall Firey's necessary and sufficient conditions \cite{Firey:12/1970} in the class of rotationally symmetric data and also Guan-Ma's sufficient condition \cite{GuanMa:03/2003} for the general data. Then we will see how the two are related.

We say a function $\psi$ defined on the unit sphere is rotationally symmetric if
\eq{\psi(x_1,\ldots,x_{n+1})=\psi(\theta),\quad x_{n+1}=\sin\theta,\quad\theta\in \left[-\frac{\pi}{2},\frac{\pi}{2}\right].}
 Note that $\theta$ is the angle that the vector from the origin to $(x_1,\ldots,x_{n+1})$ makes with $x_{n+1}=0.$ Firey has found that necessary and sufficient conditions for a function $\psi\in C^{0}(\bbS^{n})$ to admit a strictly convex solution to this problem is that in some coordinates on $\bbS^{n}$, $\psi$ is a function of the latitude $\theta$ alone, and that
\enum{
  \item $\psi$ is continuous and has finite limits as $\theta$ tends to $\pm\frac{\pi}{2}$,
  \item $\int_{\theta}^{\frac{\pi}{2}}\psi(\alpha)\cos^{n-1}\alpha\sin\alpha d\alpha>0$ and zero for $\theta=-\frac{\pi}{2},$
  \item $\psi(\theta)>\frac{n-k}{\cos^n\theta}\int_{\theta}^{\frac{\pi}{2}}\psi(\alpha)\cos^{n-1}\alpha\sin\alpha d\alpha.$
  }
In the general case, using PDE methods, a remarkable sufficient condition (but not necessary) was obtained in \cite{GuanMa:03/2003} which is based on their constant rank theorem: Suppose
\eq{\int_{\bbS^{n}}u\psi(u)d\sigma_{\bbS^{n}}=0\quad \text{and}\quad \ti D^2\psi^{-\frac{1}{k}}+\psi^{-\frac{1}{k}}\ti g\geq 0.} Then a strictly convex solution exists. Here $d\sigma_{\bbS^{n}}$ is the $n$-dimensional Hausdorff measure on $\bbS^{n}.$ The constant rank theorem ensures that if the latter matrix is non-negative definite, a convex solution is in fact strictly convex. Therefore it is sufficient to find a convex solution, for example by the continuity method or by a curvature flow approach.

Guan and Ma gave a ``natural technical explanation" of their sufficient condition. Here, through the lens of maximum principle, we provide a clear link between the constant rank theorem in the rotationally symmetric case and Firey's necessary and sufficient conditions. Let $s$ be a (spherical) convex solution. For $\theta\in (-\pi/2,\pi/2)$, define
\eq{
G(\theta)=\psi(\theta)-\frac{n-k}{\cos^n\theta}\int_{\theta}^{\frac{\pi}{2}}\psi(\alpha)\cos^{n-1}\alpha\sin\alpha d\alpha.
}
The subscript $\theta$ and superscript $'$ will denote the derivative with respect to $\theta.$

From the identity (cf., \cite{Firey:12/1970})
\eq{\label{xYz}G(\theta)=\binom{n-1}{k-1}(s''(\theta)+s(\theta))(s(\theta)-s'(\theta)\tan\theta)^{k-1},}
it follows that $G\geq 0$ and that $s$ is the support function of a strictly convex hypersurface if and only if $G>0$. We prove the latter using the strong maximum principle provided $\ti D^2\psi^{-\frac{1}{k}}+\psi^{-\frac{1}{k}}\ti g$ is non-negative definite. We calculate
\eq{
G_{\theta}&=\psi_{\theta}-n(n-k)\frac{\sin\theta}{\cos^{n+1}\theta}\int_{\theta}^{\frac{\pi}{2}}\psi(\alpha)\cos^{n-1}\alpha\sin\alpha d\alpha\\
			&\hp{=}+(n-k)\psi\tan\theta\\
&=\psi_{\theta}+nG\tan\theta-k\psi\tan\theta
}
and

\eq{G_{\theta\theta}=\psi_{\theta\theta}+n(G\tan\theta)_{\theta}-k\psi_{\theta}\tan\theta-k\psi(1+\tan^{2}\theta).}
Since $G\geq 0$ and $G(\pm \pi/2)>0,$ at any $\theta_{\star}$ with $G(\theta_{\star})=0$ we have
\eq{k\tan\theta_{\star}=\fr{\psi_{\theta}}{\psi}}
and hence
\eq{0\leq G_{\theta\theta}= \left(\psi_{\theta\theta}-\frac{k+1}{k}\frac{\psi_\theta^2}{\psi}-k\psi\right)\Big|_{\theta_{\star}}\leq 0.}
By the strong maximum principle $G\equiv 0$ which is a contradiction.


Now we find a similar sufficient condition in the rotationally symmetric case for the equation \eq{ s^{1-p}\s_{k}(\ka_{i}^{-1})=c\psi,\quad p>1.}
On the interval $-\pi/2<\theta<\pi/2,$ define
\eq{G(\theta)=s^{p-1}(\theta)\psi(\theta)-\frac{n-k}{\cos^n\theta}\int_{\theta}^{\frac{\pi}{2}}\psi(\alpha)s^{p-1}(\alpha)\cos^{n-1}\alpha\sin\alpha d\alpha.}
Again we need to show that $G> 0$. By a direct calculation,
\eq{
G_{\theta\theta}-(n-k)G_{\theta}\tan\theta
&=(n(k+1)\tan^2\theta+n)G+(\psi s^{p-1})_{\theta\theta}-k\psi s^{p-1}\\
&\hp{=}-\frac{k+1}{\psi s^{p-1}}((\psi s^{p-1})_\theta+nG\tan\theta-G_{\theta})^2.
}
Moreover, we have
\eq{
(\psi s^{p-1})_{\theta\theta}-&\frac{k+1}{k}\frac{(\psi s^{p-1})_{\theta}^2}{\psi s^{p-1}}-k\psi s^{p-1}\\
&=\left(\psi_{\theta\theta}-\frac{p+k}{p+k-1}\frac{\psi_\theta^2}{\psi}-(p+k-1)\psi\right)s^{p-1}\\
&\hp{=}-\frac{p-1}{k}\left(\frac{1}{\sqrt{p+k-1}}\frac{\psi_{\theta}}{\psi}+\sqrt{p+k-1}\frac{s_{\theta}}{s}\right)^2\psi s^{p-1}\\
&\hp{=}+(p-1)(s_{\theta\theta}+s)\psi s^{p-2}.
}
Now if $G$ attained zero at some point $\theta_{\star}$, then $(s_{\theta\theta}+s)\big|_{\theta_{\star}}=0$ and the right-hand side of this last equation would be non-positive provided
\eq{\ti D^2\psi^{-\frac{1}{p+k-1}}+\psi^{-\frac{1}{p+k-1}}\ti g\geq 0.} Hence by the strong maximum principle we would have a contradiction.
\end{rem}

\begin{rem}
By the above remark it appears for $p<1$ the condition
\eq{\ti D^2\psi^{-\frac{1}{p+k-1}}+\psi^{-\frac{1}{p+k-1}}\ti  g\geq 0}
does not ensure the strictly convexity of a solution to \eqref{CMP}.
\end{rem}
Now we state the conditions on $\p$ under which results like \Cref{KN=1} and \Cref{thm:Lp-Mink} are still valid for a wider class of curvature functions. 

\begin{defn}
Define
\eq{\G_{\e}=\{\kappa=(\ka_{i})\in \G_{+}\cn \ka_{i}\geq \e\quad\fa i\}.}
We say $F$ is in the class $\Lambda_{\varepsilon}$, if for every $\varepsilon$ there exists $C_{\varepsilon}$ such that
\eq{\ka\in \G_{\e}\quad\Ra\quad \fr{\del F}{\del \ka_{i}}\ka_{i}^{2}\leq C_{\e}F^{\g},}
where $\gamma$ is a constant which may depend on $F$.
\end{defn}

\begin{thm}\label{thm:Lp-CMP}
Let $N=\bbE^{n+1}$, $q>2$ and $\p\in C^{\8}(\bbS^{n})$ be positive.
Suppose \Cref{F2} holds. Suppose either
\enum{
\item $F$ is in the class $\Lambda_{\varepsilon}$ and \eq{\ti D^{2}\br{\p^{\fr 1q}}+\p^{\fr 1q}\ti g>0,}
\item or \eq{\ti D^{2}\br{\p^{\fr 1q}}+\fr{q-1}{q}\p^{\fr 1q}\ti g>0.}
}
Then there exists a strictly convex solution to
\eq{F=s^{1-q}\p(\nu).}
\end{thm}

In the next lemma we will verify that the class $\La_{\e}$ contains interesting curvature functions such as quotients.

\begin{lemma}
Suppose $1\leq k<\ell\leq n$. Then $F=\left(\sigma_\ell/\sigma_k\right)^{\frac{1}{\ell-k}}\in \Lambda_{\varepsilon}.$
\end{lemma}
\begin{proof}
We define
\eq{(\kappa|i)=(\kappa_1,\ldots,\kappa_{i-1},\kappa_{i+1},\ldots,\kappa_n)}
and also write
\eq{F^{ii}=\fr{\del F}{\del\ka_{i}}.}
There holds
\eq{\label{elemntary relation}
\sigma_p&=\sigma_p(\kappa|i)+\kappa_i\sigma_{p-1}(\kappa|i)=\sigma_p(\kappa|i)+\kappa_i\sigma_{p}^{ii}.}
Hence
\eq{(\ell-k)F^{ii}\ka_{i}^{2}=F\br{\fr{\s_{\ell}^{ii}\ka_{i}^{2}}{\s_{\ell}}-\fr{\s_{k}^{ii}\ka_{i}^{2}}{\s_{k}}}=F\br{\fr{(k+1)\s_{k+1}}{\s_{k}}-\fr{(l+1)\s_{l+1}}{\s_{l}}}.}
Due to \eqref{elemntary relation}, we have on $\G_{\e}$:
\eq{\frac{n}{\varepsilon}\sigma_p\geq \sum_i\sigma_{p-1}(\kappa|i)=(n-p+1)\sigma_{p-1}.}
Hence
\eq{F^{ii}\ka_{i}^{2}\leq C_{k,\ell}\fr{\s_{k+1}}{\s_{k}}F\leq C_{n,k,\ell,\e}\fr{\s_{\ell}}{\s_{k}}F=C_{n,k,\ell,\e}F^{\ell-k+1}.}
Thus $F\in \La_{\e}$ with $\g=\ell-k+1$.
\end{proof}

\begin{rem}
By the previous lemma, Condition (i) in \Cref{thm:Lp-CMP} is satisfied for the quotients
\eq{F=\br{\fr{\s_{\ell}}{\s_{k}}}^{\frac{1}{\ell-k}},}
while \Cref{F2} holds as well; see \cite[p.~23]{Andrews:/2007}.
Hence this theorem provides a generalization of \cite[Thm.~1.1]{Ivaki:02/2019} within the range $q>2$ and of \cite[Thm.~ 1.4]{GuanMaZhou:09/2006} to an important class of curvature functions (which no variational structures are available).
\end{rem}

\begin{thm}\label{thm:Lp-CMP-KN}
Let $N=\bbS^{n,1}$ and $F$ satisfy \Cref{F2}. Suppose $q<0$ and $0<\p\in C^{\8}(\bbH^{n+1})$ is a bounded function such that
\eq{\label{desitter} \p^{\fr 1q}\ti g-\ti D^{2}\br{\p^{\fr 1q}}>0.}
Then there exists a strictly convex solution to
\eq{F=cs^{1-q}\p(\nu)~\mbox{for a positive constant}~c.}
\end{thm}

\begin{rem}
We conclude this section with two remarks on the prescribed curvature problem.
\enum{
\item For $q=1$, $F=\sigma_k^{\frac{1}{k}}$ and $N=\bbS^{n,1}$, an existence result in the class of starshaped hypersurfaces was obtained recently in \cite{Ballesteros-ChavezKlingenbergLambert:08/2019}.  The reader may consult \cite{AndradeBarbosaDe-Lira:/2009,BarbosaDeLiraOliker:/2002,LiOliker:/2002,LiSheng:09/2013,SpruckXiao:05/2015} as well as \cite{Gerhardt:02/1997,Gerhardt:02/2006} regarding starshaped solutions resp. convex solutions to some classes of prescribed curvature problems in non-Euclidean ambient spaces.
\item In contrast with the Euclidean space, the Christoffel problem ($q=1$ and $F=\sigma_1(\kappa_i^{-1})$) in the hyperbolic space is a nonlinear problem.
In this case, a sufficient condition was obtained in \cite{Oliker:06/1992} employing a certain duality between the hyperbolic space and de Sitter space.
}
\end{rem}

\section{Evolution equations}\label{sec:Ev}

In this section we calculate the evolution equations for the flow
\eq{\dot{x}=\s(\phi(s,x,\ti x)-\Phi)\ti x,}
where
\eq{x\cn [0,T)\x \bbS^{n}\ra N\hra \bbR^{n+2}_{\mu}}
is viewed as a codimension $2$ embedding into Euclidean or Minkowski space, if $K_{N}\neq 0$. Here we have replaced $\nu$ by $\ti x$ to indicate that we are viewing the image of the Gauss map $\nu(\bbS^{n})$ as a closed, connected, strictly convex hypersurface in the dual space; see \cite[Ch.~9, 10]{Gerhardt:/2006}. Since $x$ and $\ti x$ range in $N$ and $\ti N$ respectively, for every $s$ we might as well extend $\phi(s,\cdot,\cdot)$ to a 0-homogeneous function:
\eq{\phi(s,x,\ti x)=\phi\br{s,\fr{x}{\abs{\abs{x}_{\mu}}},\fr{\ti x}{\abs{\abs{\ti x}_{\mu}}}}.}
When $N=\bbE^{n+1},$ we only extend $f$ with respect to $\ti x=\nu$.

For the partial derivatives of $\phi$ we use subscripts:
\eq{\phi_{s}=D_{s}\phi,\quad \phi_{x}=D_{x}\phi,\quad \phi_{\ti x}=D_{\ti x}\phi.}

For basic properties of curvature functions, the reader may consult \cite[Ch.~2]{Gerhardt:/2006} or \cite{Scheuer:06/2018}. In order to differentiate a curvature function $F$, it is convenient to view it as a function of the Weingarten operator and also as a function of the second fundamental form and the metric,
\eq{F=F(h^{i}_{j})=F(h_{ij},g_{ij}).}
We will write
\eq{F^{ij}=\fr{\del F}{\del h_{ij}},\quad F^{i}_{j}=\fr{\del F}{\del h^{j}_{i}}.}
Frequently we will use the linearized operator
\eq{\mathcal{L}:=\partial_t-\Phi'F^{ij}\nabla_i\nabla_j,}
and also use the relation
\eq{\s \bar g(\del_{r},x_{;k})=\bar g(\bar D r,x_{;k})=dr(x_{;k})=r_{;k}.}

\begin{lemma}\label{ev of speed}
There holds
\eq{\mathcal{L}(\phi-\Phi)=&\br{\s \Phi' F^{ij}h_{ik}h^{k}_{j}+ K_{N}\Phi' F^{ij}g_{ij}}(\phi-\Phi)\\
			&+\br{\s \phi_{x}(\ti x)+\s \phi_{s}\vt'-\abs{K_{N}}\ti \s \phi_{\ti x}(x)}(\phi-\Phi)\\
			&-\vt \phi_{s}(\phi-\Phi)_{;i}{r_{;}}^{i}- \phi_{\ti x}(x_{;j})(\phi-\Phi)_{;i}g^{ij},}
where $\ti \s=\bar g(\ti \nu,\ti\nu)=\ip{x}{x}_{\mu}$.
\end{lemma}

\pf{
In case $K_{N}\neq 0$, the dual flow $\ti x$ (cf., \cite[p.~29]{BIS4}) evolves by
\eq{\dot{\ti x}=\ti \s (\Phi-\phi)x+g^{ij}(\Phi-\phi)_{;i}x_{;j},}
where we have used the Weingarten equation
\eq{\ti x_{;l}=h^{k}_{l}x_{;k}.}
In case $K_{N}=0,$ we have
\eq{\dot{\nu}=g^{ij}(\Phi-\phi)_{;i}x_{;j}.}
Hence we can combine these two cases by writing
\eq{\dot{\ti x}=\abs{K_{N}}\ti \s(\Phi-\phi)x+g^{ij}(\Phi-\phi)_{;i}x_{;j}.}
By \cite[Lem.~2.3.3]{Gerhardt:/2006} we have
\eq{\label{Ev-A-gen}\dot{h}^{j}_{i}&=-{(\phi-\Phi)_{;i}}^{j}-\sigma(\phi-\Phi)h_{ik}h^{kj}- K_N(\phi-\Phi)\delta_{i}^{j},\\
\del_{t}(\phi-\Phi)&=\del_{t}\phi-\Phi'F_j^i\dot{h}_i^j.}
Using that $\vt\del_{r}$ is a conformal Killing field, i.e.,
\eq{\label{Conf}\br{\vt \del_{r}}_{;\al}=\vt'\del_{\al},}
and \cite[Lem.~2.3.2]{Gerhardt:/2006} we calculate
\eq{\label{s-dot}\dot{s}&=\s\del_{t}\bar g(\vt\del_{r},\nu)\\
&=\s\vt'\bar g(\dot x,\nu)+\s\vt\bar g(\del_{r},\dot{\nu})\\
&=\s\vt'(\phi-\Phi)-\s(\phi-\Phi)_{;i}\bar g(\vt\del_{r},x_{;j})g^{ij}.}
Hence
\eq{\label{time derv s}\del_{t}\phi&= \phi_{x}(\dot{x})+\phi_{s}\dot{s}+\phi_{\ti x}(\dot{\ti x})\\
		&=\s(\phi-\Phi)\phi_{x}(\ti x)+\s \phi_{s}\vt'(\phi-\Phi)-\s \phi_{s}(\phi-\Phi)_{;i}\bar g(\vt \del_{r},x_{;j})g^{ij}\\
			&\hp{=}- \phi_{\ti x}(x_{;j})(\phi-\Phi)_{;i}g^{ij}-\abs{K_{N}}\ti \s \phi_{\ti x}(x)(\phi-\Phi).}
The proof is concluded by tracing $\dot{A}$ with respect to $F^{ij}$.
}

\begin{lemma}\label{Ev-h}
The components of the Weingarten operator evolve by
\eq{\mathcal{L}h_i^j&=\s\Phi'F^{kl}h_{mk}h^{m}_{l}h^{j}_{i}-\s(\Phi'F+\phi-\Phi)h^{j}_{m}h^{m}_{i}+K_{N}(\Phi-\phi+\Phi'F)\de^{j}_{i}\\
				&\hp{=}-K_{N}\Phi'F^{kl}g_{kl}h^{j}_{i}+\Phi'F^{kl,rs}h_{kl;i}{h_{rs;}}^{j}+\Phi'' F_{;i}{F_{;}}^{j}\\
				&\hp{=}-\phi_{xx}(x_{;i},x_{;k})g^{kj}-\vt\phi_{xs}(x_{;i}) r_{;k}h^{kj}-\phi_{x\ti x}(x_{;i},x_{;k})h^{kj}\\
				&\hp{=}+\s \phi_{x}(\ti x)h^{j}_{i}-\vt\phi_{sx}(x_{;m})r_{;k}h^{k}_{i}g^{mj}-\vt^{2}\phi_{ss}r_{;k}r_{;m}h^{k}_{i}h^{mj}\\
		&\hp{=}-\vt\phi_{s\ti x}(x_{;l})r_{;k}h^{k}_{i}h^{lj}+\s\phi_{s} h_{i}^{k}h_{k}^{j}s-\s \phi_{s} \vt'h_{i}^{j}-\vt\phi_{s}r_{;k}{h^{k}_{i;m}}g^{mj}\\
		&\hp{=}-\phi_{\ti x x}(x_{;k},x_{;m})h^{k}_{i}g^{mj}-\vt\phi_{\ti xs}(x_{;l})r_{;k}h^{l}_{i}h^{kj}-\phi_{\ti x\ti x}(x_{;k},x_{;l})h^{k}_{i}h^{lj}\\
		&\hp{=}-\phi_{\ti x}(x_{;k})h^{kj}_{i;}+K_{N}\phi_{\ti x}(x)h_{i}^{j}.}
\end{lemma}

\pf{
We start from
\eq{\label{Ev-A-gen-b}\dot{h}^{j}_{i}&=-{(\phi-\Phi)_{;i}}^{j}-\sigma(\phi-\Phi)h_{ik}h^{kj}-K_N(\phi-\Phi)\delta_{i}^{j}.}
First we have to replace the term ${\Phi_{;i}}^{j}$. There holds
\eq{{\Phi_{;i}}^{j}={\Phi'F^{kl}h_{kl;i}}^{j}+\Phi'F^{kl,rs}h_{kl;i}{h_{rs;}}^{j}+\Phi''F_{;i}{F_{;}}^{j}.}
Now we use the Codazzi  and Gauss equation to deduce
\eq{h_{kl;ij}&=h_{ki;lj}\\
		&=h_{ki;jl}+{R_{ljk}}^{m}h_{mi}+{R_{lji}}^{m}h_{mk}\\
		&=h_{ij;kl}+\s(h_{jk}h^{m}_{l}h_{mi}-h_{lk}h^{m}_{j}h_{mi})+\ov \Rm(x_{;l},x_{;j},x_{;k},x_{;m})h^{m}_{i}\\
		&\hp{=}+\s(h_{ji}h^{m}_{l}h_{mk}-h_{li}h^{m}_{j}h_{mk})+\ov \Rm(x_{;l},x_{;j},x_{;i},x_{;m})h^{m}_{k}\\
		&=h_{ij;kl}+\s(h_{jk}h^{m}_{l}h_{mi}-h_{lk}h^{m}_{j}h_{mi})+K_{N}(g_{lm}g_{jk}-g_{lk}g_{jm})h^{m}_{i}\\
		&\hp{=}+\s(h_{ji}h^{m}_{l}h_{mk}-h_{li}h^{m}_{j}h_{mk})+K_{N}(g_{lm}g_{ij}-g_{il}g_{jm})h^{m}_{k}\\
		&=h_{ij;kl}+\s(h_{jk}h^{m}_{l}h_{mi}-h_{lk}h^{m}_{j}h_{mi})+K_{N}(h_{il}g_{jk}-g_{lk}h_{ij})\\
		&\hp{=}+\s(h_{ji}h^{m}_{l}h_{mk}-h_{li}h^{m}_{j}h_{mk})+K_{N}(h_{lk}g_{ij}-g_{il}h_{jk}).}
Using $F^{k}_{l}h^{l}_{m}=h^{k}_{l}F^{l}_{m}$ we get
\eq{\Phi'F^{kl}h_{kl;ij}&=\Phi'F^{kl}h_{ij;kl}-\s\Phi'F^{kl}h_{lk}h^{m}_{j}h_{mi}+K_{N}\Phi'F^{kl}(h_{il}g_{jk}-g_{lk}h_{ij})\\
				&\hp{=}+\s\Phi'F^{kl}h_{ij}h^{m}_{l}h_{mk}+K_{N}\Phi'F^{kl}(h_{lk}g_{ij}-g_{il}h_{jk}).}
From the $1$-homogeneity of $F$ it follows that
\eq{\Phi'F^{kl}h_{kl;ij}&=\Phi'F^{kl}h_{ij;kl}-\s\Phi'Fh^{m}_{j}h_{mi}-K_{N}\Phi'F^{kl}g_{kl}h_{ij}\\
				&\hp{=}+\s\Phi'F^{kl}h^{m}_{l}h_{mk}h_{ij}+K_{N}\Phi'Fg_{ij}.}
Inserting this into \eqref{Ev-A-gen-b} we obtain
\eq{\label{Ev-h-1}
\cL h^{j}_{i}&=-\Phi'F^{kl}h^j_{i;kl}+{\Phi'F^{kl}h_{kl;i}}^{j}+\Phi'F^{kl,rs}h_{kl;i}{h_{rs;}}^{j}\\
			&\hp{=}+\Phi''F_{;i}{F_{;}}^{j}-{\phi_{;i}^{j}}-\s(\phi-\Phi)h_{ik}h^{kj}-K_{N}(\phi-\Phi)\de^{j}_{i}\\
			&=\s\Phi'F^{kl}h^{m}_{l}h_{mk}h^{j}_{i}-\s(\Phi'F+\phi-\Phi)h_{i}^{m}h^{j}_{m}\\
			&\hp{=}+K_{N}(\Phi-\phi+\Phi'F)\de^{j}_{i}-K_{N}\Phi'F^{kl}g_{kl}h^{j}_{i}\\
			&\hp{=}+\Phi'F^{kl,rs}h_{kl;i}{h_{rs;}}^{j}+\Phi''F_{;i}{F_{;}}^{j}-{\phi_{;i}^{j}}.}
Now we calculate the term $\phi_{;i}^{j}.$ Using the codimension $2$ Gaussian formula (cf., \cite[Ch.~9, 10]{Gerhardt:/2006}),
\eq{x_{;kj}=-\s h_{kj}\ti x-K_{N}g_{kj}x,}
we calculate
\eq{{\phi_{;i}}=\phi_{x}(x_{;i})+\phi_{s}s_{;i}+\phi_{\ti x}(\ti x_{;i})}
and
\eq{\label{derv tilde f}\phi_{;ij}&=-\s \phi_{x}(\ti x)h_{ij}-K_{N}\phi_{x}(x)g_{ij}\\
        &\hp{=}+\phi_{xx}(x_{;i},x_{;j})+\phi_{xs}(x_{;i})s_{;j}+\phi_{x\ti x}(x_{;i},\ti x_{;j})\\
		&\hp{=}+\phi_{sx}(x_{;j})s_{;i}+\phi_{ss}s_{;i}s_{;j}+\phi_{s\ti x}(\ti x_{;j})s_{;i}+\phi_{s}s_{;ij}\\
		&\hp{=}+\phi_{\ti xx}(\ti x_{;i},x_{;j})+\phi_{\ti xs}(\ti x_{;i})s_{;j}+\phi_{\ti x\ti x}(\ti x_{;i},\ti x_{;j})+\phi_{\ti x}(\ti x_{;ij}).}
From \eqref{Conf} we obtain
\eq{s_{;i}&=\s\del_{i}\bar g(\vt\del_{r},\nu)=\s\bar g(\vt\del_{r},x_{;k})h^{k}_{i}=\vt r_{;k}h^{k}_{i}\\
s_{;ij}&=\s\vt'h_{ij}-\s h^{k}_{i}h_{kj}s+\vt r_{;k}h^{k}_{i;j}.}
Moreover, by the Weingarten equation we have
\eq{\ti x_{;i}=h^{k}_{i}x_{;k},\quad \ti x_{;ij}=h^{k}_{i;j}x_{;k}-\s h^{k}_{i}h_{kj}\ti x-K_{N}h_{ij}x.}
 The result follows from substituting the expression for $\phi_{;ij}$ into \eqref{Ev-h-1} and using the zero homogeneity in $x$ and $\ti x$, i.e.,
\eq{K_N\phi_{x}(x)=\phi_{\ti x}(\ti x)=0.}
}

Let $(b_r^s)$ denote the inverse of $(h_r^s).$ The next lemma relates their evolution equations.

\begin{lemma}\label{Ev-b}
The evolution equations of $(h^{i}_{j})$ and $(b^{i}_{j})$ are related by
\eq{\mathcal{L}b^{r}_{s}&=-b^{i}_{s}b^{r}_{j}\mathcal{L}h_i^j-2\Phi'F^{kp}b^{lq}h_{kl;i}h_{pq;j}b^{j}_{s}b^{ri}.}
\end{lemma}

\pf{
According to the rule on how to differentiate the inverse of a matrix, there hold
\eq{\dot{b}^{r}_{s}=-b^{r}_{j}\dot{h}^{j}_{i}b^{i}_{s},\quad {b}^{r}_{s;k}=-b^{r}_{i}{h}^{i}_{l;k}b^{l}_{s}}
and
\eq{b^{r}_{s;kp}=-b^{r}_{i}h^{i}_{l;kp}b^{l}_{s}+b^{r}_{m}h^{m}_{q;p}b^{q}_{i}h^{i}_{l;k}b^{l}_{s}+b^{r}_{i}h^{i}_{l;k}b^{l}_{q}h^{q}_{j;p}b^{j}_{s}.}
Hence
\eq{\mathcal{L}b^r_s = \dot{b}^{r}_{s}-\Phi'F^{kl}b^{r}_{s;kl}=-b^{r}_{j}\cL h^{j}_{i}b^{i}_{s}-2\Phi'F^{kp}b^{r}_{i}h^{i}_{l;k}b^{l}_{q}h^{q}_{j;p}b^{j}_{s}.}
Substituting in the evolution of $h$ from \Cref{Ev-h} and using that $b$ is the inverse of $h$ gives the result.
}

\begin{lemma}
The function $\vt'(r)$ satisfies
\eq{\cL \vt'=K_{N}(\phi-\Phi+\Phi'F)\ip{\ti x}{e_{n+2}}_{\mu}+K_{N}\Phi'F^{ij}g_{ij}\vt'.}
\end{lemma}

\pf{We may obtain $\cL\vt'$ with the following trick without calculating the evolution equation of $r.$
The intrinsic radial distance to the north pole of the sphere, the Beltrami point in the hyperbolic space, or the totally umbilic slice of de Sitter space is related to the codimension 2 embedding vector via the relation
\eq{\vt'(r)=\mu\ip{x}{e_{n+2}}_{\mu}.}
See for example \cite[Chapter 10, particularly 10.3]{Gerhardt:/2006}
Hence
\eq{\del_{t}\vt'=\mu\ip{\dot{x}}{e_{n+2}}_{\mu}=\mu\s(\phi-\Phi)\ip{\ti x}{e_{n+2}}_{\mu},}
\eq{\vt'_{;ij}=\mu\ip{x_{;ij}}{e_{n+2}}_{\mu}=-\mu\s h_{ij}\ip{\ti x}{e_{n+2}}_{\mu}-K_{N}g_{ij}\vt'.}
The result follows from combining these equalities and $\mu\s=K_{N}$ in case $K_{N}\neq 0$, while the lemma is trivial in case $K_{N}=0$, where $\vt'=1$.
}

%

\section{Proofs of Theorems}\label{sec:pf}

The following gradient bound for (weakly) convex hypersurfaces (i.e., $\ka_{i}\geq 0$) can be found in \cite[Thm.~2.7.10, 2.7.11]{Gerhardt:/2006}. Recall the definition of $v$ from \eqref{v}.
	
\begin{lemma}[Bounds to first order]\label{grad}
Let $\Om\sub N$ be as in \Cref{N} and \Cref{Om} and $f\in C^{\8}(\bbR_{+}\x \bar\Om\x \ti N).$
Then every spacelike weakly convex hypersurface $\Si\sub \bar\Om$ satisfies:
\enum{
\item For some positive constant $C=C(\Om)$,
\eq{v+\fr{1}{v}\leq C.}
In particular, there exists a constant $C=C(\Om)>0$, such that
\eq{C^{-1}\leq s\leq C.}
\item For all $m\geq 0$ there exists $C=C(m,\Om)$ such that
\eq{\abs{(D^{m}f)_{|\Si}}\leq C,}
where we measure the norm of tensors on a Lorentzian space with respect to the natural Riemannian background metric.
}
\end{lemma}

\pf{
To prove (i), note that in case $\s=1$ we have $v\geq 1$ and the upper bound of $v$ follows from \cite[Thm.~2.7.10]{Gerhardt:/2006}. In case $\s=-1$ we have $v\leq 1$ and the bound of $v^{-1}$ follows from \cite[Thm.~2.7.11]{Gerhardt:/2006}. Since
\eq{s=\fr{\vt}{v}}
and $\vt>0$ in $\bar\Om$, the bounds on $s$ follow as well. To prove (ii), note that from (i) and the assumption on $\Si$, the first two variables of $f$ range in a compact subset of $\bbR_{+}\x\bar\Om$. The proof will be complete once we show the image of $\Si$ under the Gauss map ranges in a compact subset of $\ti N$. Only when $\ti N=\bbS^{n,1}$ or $\ti N=\bbH^{n+1}$ the set $\ti N$ is non-compact. Since we have barriers and bounds on $v$ and $v^{-1}$, we can use equations (10.4.65) and (10.4.67) in the proof of \cite[Thm.~10.4.9]{Gerhardt:/2006} to see that the Gauss maps also enjoy barriers depending only on $\Om$.
}

\subsection{Proof of {\texorpdfstring{\Cref{Flow-Main}}{Thm.~3.2}}}\label{pf:Flow-Main}

We recall the following inequality; see e.g., \cite[p.~112]{Urbas:/1991}. If $F\in C^{\8}(\G_{+})$ is inverse concave, then
\eq{\label{invers concav}(F^{kl,pq}+2F^{kp}b^{lq})\eta_{kl}\eta_{pq}\geq \fr{2}{F}(F^{kl}\eta_{kl})^{2}}
for all symmetric matrices $(\eta_{kl})$.

Now we start deriving with the curvature estimates.

\begin{lemma}
Under the assumptions of \Cref{Flow-Main} and along the flow \eqref{Flow-Main-2} there exists a positive constant $c$ depending only on the data of the problem, such that
\eq{\label{k1}\ka_{i}\geq c.}
\end{lemma}

\pf{
Let $T^{*}>0$ be the maximal time of smooth existence (see \cite[Sec.~2.5, 2.6]{Gerhardt:/2006} for the existence of $T^{*}$) and $0<T_{\star}\leq T^{*}$ be the supremum of all times up to which the flow is strictly convex.
To prove the lemma, we show that \eqref{k1} holds up to $T_{\star}$ with a constant $c>0$ that only depends on the data of the problem. This then shows that $T_{\star}=T^{*}$ and concludes the proof of the lemma.

Due to our assumptions about the presence of barriers, by \Cref{grad} all derivatives of $\phi$ are uniformly bounded up to $T_{\star}$. Since $M_{t}$ is strictly convex for $t<T_{\star}$, by \Cref{ev of speed} $M_{t}$  is a lower barrier ($\s=1$) or an upper barrier ($\s=-1$) and hence
\eq{F\leq f\leq C.}
We use \Cref{Ev-b} to find the evolution of
\eq{B:=g_{rs}b^{rs}=b^{r}_{r}}
and use the bounds on $\phi$ and its derivatives and \eqref{invers concav} with
 \eq{\eta_{kl}=h_{kl;r}\ka_{r}^{-1}}
  to estimate
\eq{\mathcal{L}B&\leq-\fr{2}{F^{2}}F^{kp}b^{lq}h_{kl;i}h_{pq;j}b^{i}_{r}b^{rj}-\fr{1}{F^{2}}F^{kl,pq}h_{kl;i}h_{pq;j}b^{i}_{r}b^{rj}\\
				&\hp{=}+\fr{2}{F^{3}}F^{kl}F^{pq}h_{kl;i}h_{pq;j}b^{i}_{r}b^{rj}-\fr{\s}{F^{2}}F^{kl}h_{rk}h^{r}_{l}B\\
				&\hp{=}+n\s\br{\fr{2}{F}+\phi}+K_{N}\phi \abs{b}^{2}+\fr{K_{N}}{F^{2}}F^{kl}g_{kl}B\\
				&\hp{=}+\phi_{xx}(x_{;i},x_{;j})b^{i}_{r}b^{rj}+C(B+1)-\phi_{s}\vt B_{;k}{r_{;}}^{k}-\phi_{\ti x}(x_{;k}){B_{;}}^{k}		\\
				&\leq \sum_{r=1}^n\br{\phi_{xx}+K_N \phi \bar g}(x_{;r}\ka_{r}^{-1},x_{;r}\ka_{r}^{-1})-\fr{\s}{F^{2}}F^{kl}h_{rk}h^{r}_{l}B\\
				&\hp{=}+n\s\br{\fr{2}{F}+\phi}+\fr{K_{N}}{F^{2}}F^{kl}g_{kl}B+C(B+1)\\
				&\hp{=}-\phi_{s}\vt B_{;k}{r_{;}}^{k}-\phi_{\ti x}(x_{;k}){B_{;}}^{k}.}
For the case $ K_{N}\leq 0$ (which also implies $\s=1$) a bound on $B$ up to $T_{\star}$ follows easily: Due to
$\fr{1}{F}\leq B$, the strict inequality assumption \eqref{Flow-Main-1} on the Hessian of $\phi$ and that in view of \Cref{grad} the arguments of $\phi$ range in a compact set, a good second degree term dominates the right-hand side.

In case $K_{N}=1$, we use the concavity of $F$ and \cite[Lem.~2.2.19]{Gerhardt:/2006} when $N=\bbS^{n+1}$, and also the uniform monotonicity in case $N=\bbS^{n,1}$ to ensure that up to $T_{\star}$ there holds
\eq{\label{c_F} F^{ij}g_{ij}\geq c_F>0}
for some constant $c_{F}$.
For $\de$ sufficiently small define
\eq{w=\log B+\al(\vt'),~\mbox{where}~\al(\vt'):=-\log(\vt'-\de).}
Then $w$ satisfies
\eq{\cL w&\leq  -\e B-\fr{\s}{F^{2}}F^{kl}h_{rk}h^{r}_{l}+\fr{n\s}{B}\br{\fr{2}{F}+\phi}+C(1+B^{-1})\\
		&\hp{=}+\fr{1}{F^{2}}F^{kl}g_{kl}-\vt\phi_{s}r_{;k}{(\log B)_{;}}^{k}-\phi_{\ti x}(x_{;k}){(\log B)_{;}}^{k}\\
		&\hp{=}+\fr{1}{F^{2}}F^{kl}(\log B)_{;k}(\log B_{;l})+\al'\cL\vt'-\fr{\al''}{F^{2}}F^{kl}\vt'_{;k}\vt'_{;l}.}
With the help of
\eq{\al''=\al'^{2},\quad 1+\al'\vt'=-\fr{\de}{\vt'-\de}=\de\al',\label{alpha properties}}
at a maximum point we obtain
\eq{\cL w&\leq  -\e B-\fr{\s}{F^{2}}F^{kl}h_{rk}h^{r}_{l}+\fr{n\s}{B}\br{\fr{2}{F}+\phi}+C(1+B^{-1})\\
		&\hp{=}+\de\al'\fr{1}{F^{2}}F^{kl}g_{kl}+C\abs{\al'}\abs{\n \vt'}+C\abs{\al'}(1+F^{-1}).}
We may use the resulting good, strictly negative $1/F^{2}$ term (from equation \eqref{c_F}) to absorb $C/F$ up to a constant.
In case $\s=-1,$ due to convexity of $F$ we have
\eq{F^{kl}h_{rk}h^r_l\leq FH\leq F^2,} while in case $\s=1$ the term involving this expression is negative.
Thus $B$ is uniformly bounded.}

We finish the proof of \Cref{Flow-Main}. By the previous lemma, uniform curvature estimates follow from $F\leq f$ and $F_{\ast|\del\G_{+}}=0$. In fact, we have
\eq{C\geq \fr{F}{\ka_{1}}=F\br{1,\dots,\fr{\ka_{n}}{\ka_{1}}}=\fr{1}{F_{\ast}(\tfrac{\ka_{1}}{\ka_{n}},\dots,1)};}
hence, if $\ka_{n}\to\infty$, then the right-hand side would blow up as well.

To deduce $C^{2,\al}$-estimates by \cite{Krylov:/1987}, we need to work around the obstruction that the curvature function is not concave.
Since we are working in simply connected spaceforms, we distinguish three cases. If $K_{N}=0$, we use the Gauss map parametrization, under which the support function satisfies
\eq{\label{gauss map derv s}\dot{s}&=\ip{\dot x}{\nu}=(\phi-\Phi)=\br{\fr{1}{F(h^{j}_{i})}+\phi}=\br{F_{\ast}(\ti h^{j}_{i})+\phi}.}
Here $\wt \cW=(\ti h^{j}_{i})$ is the inverse of the Weingarten map; i.e., in terms of the round metric $\bar g$,
\eq{\wt \cW=(\hat\n^2 s)^{\sharp}+s\hat g.}
Since $F_{\ast}$ is concave, \eqref{gauss map derv s} satisfies the assumptions of the Krylov-Safonov theorem \cite{Krylov:/1987} and the higher order regularity estimates for the support function follow. Due to strict convexity, we also obtain higher order regularity estimates for the original flow. In the other cases, we can use the dual flow method to obtain regularity of support function as in \cite{Gerhardt:/2015,Wei:04/2019}.

From the $C^{\8}$-estimates and the monotonicity of the flow we obtain a unique smooth limit hypersurface.
The radial function $r$ of the flow hypersurfaces (cf., \cite[p.~98-99]{Gerhardt:/2006}) satisfies
\eq{\fr{\del r}{\del t}=\s \br{\phi+\fr{1}{F}}v.}
Integration and using that $\phi+F^{-1}$ is non-negative gives
\eq{\abs{r(t,x)-r(0,x)}= \int_{0}^{t}v\br{\fr{1}{F}+\phi}.}
Due to \Cref{grad},
\eq{\int_{0}^{\8}\left|\fr{1}{F}+\phi\right|\leq c\int_{0}^{\8}v\left(\fr{1}{F}+\phi\right)<\8.}
Therefore, in view of the monotonicity of $r$, the limit
\eq{\ti r(x):=\lim_{t\ra \8}r(t,x)}
exists and is smooth and it must satisfy \eqref{GMP}.

\subsection{Proofs of {\texorpdfstring{\Cref{thm:Lp-Mink}}{Thm.~3.7}} and {\texorpdfstring{\Cref{thm:Lp-CMP}}{Thm.~3.13}}}\label{pf:Lp-Mink}

In the following lemma, we give a crucial estimate that is needed for treating the wider class of curvature functions in \Cref{thm:Lp-CMP}, while for \Cref{thm:Lp-Mink} a crude bound suffices.

\begin{lemma}\label{Hessian}
Let
\eq{f\cn \bbR_{+}\x \ti N&\ra \bbR\\
				(s,\ti x)&\mt s^{1-q}\p(\ti x)}
 with $q(q-1)>0$, where $\p\in C^{\8}(\ti N)$ is extended as a degree zero function as in \Cref{sec:Ev}. Then we have
\eq{-\vt^{2} f_{ss}r_{;n}^{2}-&2\vt f_{s\ti x}(x_{;n}){r_{;n}}- f_{\ti x\ti x}(x_{;n},x_{;n})\\
&\leq-qs^{1-q}\p^{1-\fr 1q}\ti D^{2}(\p^{\fr 1q})(x_{;n},x_{;n}).}
\end{lemma}
\pf{ Let us put
\[Q=\vt^{2} f_{ss}r_{;n}^{2}+2\vt f_{s\ti x}(x_{;n}){r_{;n}}+ f_{\ti x\ti x}(x_{;n},x_{;n}).\]
We calculate
\eq{f_{ss}=q(q-1)s^{-(q+1)}\p,\quad f_{s\ti x}=(1-q)s^{-q}\p_{\ti x},\quad f_{\ti x\ti x}=s^{1-q}\p_{\ti x\ti x.}}
Hence choosing
$\zeta=\sgn(q-1)\fr{q\p}{s},$
and using Young's inequality with $\epsilon = \zeta$ we get
\eq{Q&=-2(q-1)s^{-q}\p_{\ti x}(x_{;n})\vt r_{;n}\\
&\hp{=}+q(q-1)s^{-(q+1)}\p\vt^{2}r_{;n}^{2}+s^{1-q}\p_{\ti x\ti x}(x_{;n},x_{;n})\\
&\geq-\zeta\abs{q-1}s^{-q}\vt^{2}r_{;n}^{2}-\fr{\abs{q-1}}{\zeta}s^{-q}\p_{\ti x}(x_{;n})^{2}\\
&\hp{=}+q(q-1)s^{-(q+1)}\p\vt^{2}r_{;n}^{2}+s^{1-q}\p_{\ti x\ti x}(x_{;n},x_{;n})\\
&=qs^{1-q}\p^{1-\fr 1q}\br{\p^{\fr 1q}}_{\ti x\ti x}(x_{;n},x_{;n})\\
&=qs^{1-q}\p^{1-\fr 1q}\ti D^{2}(\p^{\fr 1q})(x_{;n},x_{;n}).}
Here we employed the zero homogeneous extension of $\p$ to infer that the full Hessian of $\p$ equals the Hessian on $\ti N$.
}
To prove \Cref{thm:Lp-Mink,thm:Lp-CMP}, it is favorable to use a contracting type flow, hence we choose
\eq{\Phi(F)=F,\quad \phi=f=s^{1-q}\p(\ti x).}
Due to the scaling properties of $F$ and $s$, and the range of $q$, we have spherical barriers. We start from a lower barrier.
Let us state the simplified version of the evolution equation of the second fundamental form. Using \Cref{Ev-h} we find

\eq{\label{pf:Lp-Mink-1}\mathcal{L}h_i^j&=F^{kl}h_{kr}h^{r}_{l}h^{j}_{i}- f h^{j}_{k}h^{k}_{i}+F^{kl,rs}h_{kl;i}{h_{rs;}}^{j}-\vt^{2} f_{ss} h_{il}{r_{;}}^{l} h_{k}^{j}{r_{;}}^{k}\\
		&\hp{=}-\vt f_{s\ti x}(x_{;l})h_{ik}{r_{;}}^{k}h^{lj}+ f_{s} h_{i}^{k}h_{k}^{j}s-  f_{s} h_{i}^{j}- f_{s}\vt {h_{ik;}}^{j}{r_{;}}^{k}\\
		&\hp{=}-\vt f_{\ti xs}(x_{;l})h^{l}_{i} h_{k}^{j}{r_{;}}^{k}- f_{\ti x\ti x}(x_{;k},x_{;l})h^{k}_{i}h^{lj}- f_{\ti x}(x_{;k})h^{kj}_{i;}.}
To obtain a lower bound on the principal curvatures in both theorems, we proceed as in the proof of \Cref{Flow-Main}. Again let $T_{\star}\leq T^*$ be the supremum of all times up to which the flow is strictly convex. Suppose $T_{\star}<T^*$ and define
\eq{B=g_{rs}b^{rs}=b^{r}_{r}}
as before. We will prove a uniform upper bound on $B$ up to $T_{\star}$, which then contradicts the definition of $T_{\star}$, unless $T_{\star}=T^*$. This will imply preservation of strict convexity and the existence of a lower bound
\eq{\ka_1\geq \e>0,}
where $\e$ only depends on the data of the problem.

Now we estimate $B$ using \Cref{Ev-b}. Up to $T_{\star}$ we have
\eq{\mathcal{L}B&\leq -F^{kl}h_{kr}h^{r}_{l}B+f_{s}B+C-f_{s}\vt B_{;k}{r_{;}}^{k}-f_{\ti x}(x_{;k}){B_{;}}^{k}\\
					&\leq (1-q)s^{-q}\p(\ti x)B+C-f_{s}\vt B_{;k}{r_{;}}^{k}-f_{\ti x}(x_{;k}){B_{;}}^{k},}
where we have used \cref{grad} and \eqref{invers concav}. Since $q>2$, the first degree $B$-term dominates the right-hand side and hence $B$ is uniformly bounded up to $T_{\star}$.

Since we started from a lower barrier, $F\leq C.$ Hence in the case of \Cref{thm:Lp-Mink}, using $F_{\ast|\del\G_{+}}=0$ we can obtain uniform curvature bounds, and the proof can be completed using the Gauss map parametrization.

Regarding \Cref{thm:Lp-CMP}, we need further arguments to obtain upper curvature bounds. We work in a local coordinate system, such that at a maximum point of $\ka_{n}$ in space-time
\eq{g_{ij}=\de_{ij},\quad h_{ij}=\ka_{i}\de_{ij}.}
At such a point, we obtain after dividing \eqref{pf:Lp-Mink-1} by $\ka_{n}^{2}$,
\eq{0&\leq C\ka_{n}^{-1}+\ka_{n}^{-1}F^{kl}h_{rk}h^{r}_{l}-qf-qf\p^{-\fr 1q}\ti D^2(\p^{\fr 1q})(x_{;n},x_{;n}),
}
where we used \Cref{Hessian}.
Regarding case (i), we use that $F$ is of class $\La_{\e}$ and the bound on $F$ to conclude that the strictly negative term on the right-hand side dominates and we obtain a contradiction for large $\ka_{n}$. In case (ii), using the $1$-homogeneity of $F$ we estimate
\eq{F^{kl}h_{rk}h^{r}_{l}\leq F\ka_{n}\leq f\ka_{n}}
and again the condition on the spherical Hessian gives a bound on $\ka_{n}$.
With the uniform curvature estimates at hand, the proof can now be completed as in \Cref{pf:Flow-Main} using the Gauss map parametrization.

\subsection{Proof of {\texorpdfstring{\Cref{thm:Lp-CMP-KN}}{Theorem 3.16}}}

Recall that we have
\eq{\s=-1,\quad K_{N}= 1,\quad f=cs^{1-q}\p(\nu),}
where, as in the proof of \Cref{KN=1}, we choose a constant $c$ as follows to ensure the existence of barriers. We have
\eq{\bar\ka_{i}=\fr{\vt'}{\vt}=\tanh r.}
Let $a>0$. Using that $\p$ is bounded we can pick $c$ such that
\eq{0<c<\inf_{\nu\in \bbH^{n+1}}\fr{n\sinh a}{\p(\nu)\cosh^{2-q}a}.}
This yields
\eq{(F-f)_{|r=a}=n\tanh a-c\p\cosh^{1-q}a>0.}
Moreover, we have
\eq{\limsup_{r\ra \8}\br{n\tanh r-c\p\cosh^{1-q}r}<0.}
Therefore, defining $\Om=(a,b)\x\bbS^{n}$ with $a$ as above and $b$ sufficiently large, we obtain barriers for this curvature problem in $\bbS^{n,1}$. Then we start the flow from a lower barrier, where $\Phi=F$. Hence
\eq{F\geq f\geq  c.\label{lower F bound}}
The evolution of the second fundamental form in \Cref{Ev-h} becomes
\eq{\mathcal{L}h_i^j&=-F^{kl}h_{kr}h^{r}_{l}h^{j}_{i}+f h^{j}_{k}h^{k}_{i}+(2F-f)\de^{j}_{i}-F^{kl}g_{kl}h^{j}_{i}+F^{kl,rs}h_{kl;i}{h_{rs;}}^{j}\\
				&\hp{=}-\vt^{2}f_{ss} h_{il}{r_{;}}^{l} h_{k}^{j}{r_{;}}^{k}-\vt f_{s\ti x}(x_{;l})h_{ik}{r_{;}}^{k}h^{lj}-f_{s} h_{i}^{k}h_{k}^{j}s+ f_{s} \vt'h_{i}^{j}\\
				&\hp{=}-f_{s}\vt {h_{ik;}}^{j}{r_{;}}^{k}-\vt f_{\ti xs}(x_{;l})h^{l}_{i} h_{k}^{j}{r_{;}}^{k}-f_{\ti x\ti x}(x_{;k},x_{;l})h^{k}_{i}h^{lj}\\
		&\hp{=}-f_{\ti x}(x_{;k})h^{kj}_{i;}+f_{\ti x}(x)h_{i}^{j}.}
Let $T_{\star}\leq T^{*}$ be as in the previous proofs. From the concavity of $F$, \cref{grad}, \Cref{Hessian} and \eqref{desitter} it follows that up to $T_{\star}$ and at a maximum point of $\ka_{n}$ with $\ka_{n}\geq 1$ we have
\eq{0&\leq -\ka_{n}^{-1}F^{kl}h_{kr}h^{r}_{l}+ qf+C\ka_{n}^{-1}-qf\p^{-\fr 1q}\ti D^{2}(\p^{\fr 1q})(x_{;n},x_{;n})\\
	&\leq -\ka_{n}^{-1}F^{kl}h_{kr}h^{r}_{l}+ \e qf+C\ka_{n}^{-1},}
where we used $q<0$ and that for some $\e>0,$
\eq{\p^{-\fr 1q}\ti D^{2} (\p^{\fr 1q})<(1-\e)\ti g.}

Therefore, in view of $q<0$ and the uniform positivity of $f$, principal curvatures are uniformly bounded above up to $T_{\star}$.

Now we prove a uniform lower curvature bound up to $T_{\star}$. Define $B$ as above and again use \Cref{Ev-b}. Using $F\leq n\ka_{n}\leq C$, up to $T_{\star}$ there holds
\eq{\mathcal{L}B&\leq C(B+1)-F\abs{b}^{2}+F^{kl}g_{kl}B-f_{s}\vt B_{;k}{r_{;}}^{k}-f_{\ti x}(x_{;k}){B_{;}}^{k}.  }
We will deal with the term $F^{kl}g_{kl}$ by using the radial function.
Define
\eq{w=\log B+\al(\vt'),}
where $\al$ is as in the proof of \Cref{Flow-Main}.
Then at any maximum point with $B\geq 1$ we have
\eq{\mathcal{L}w&\leq C(1+\abs{\al'})-\fr{F}{n}B+\br{1+\al'\vt'}F^{kl}g_{kl}\\
				&\hp{=}+\br{\al'^{2}-\al''}F^{kl}\vt'_{;k}\vt'_{;l}.}
In view of \eqref{alpha properties} and \eqref{lower F bound}, the $B$-term dominates the right-hand side and the bound on $B$ follows. Hence $B$ is bounded up to $T_{\star}$, which implies $T_{\star}=T^{*}$ and in turn we obtain uniform upper and lower curvatures bounds up to $T^{*}$. Now the proof can be completed as in the previous theorems.
\section*{Acknowledgment}
PB was supported by the ARC within the research grant ``Analysis of fully non-linear geometric problems and differential equations", number DE180100110. MI was supported by a Jerrold E. Marsden postdoctoral fellowship from the Fields Institute. JS was supported by the ``Deutsche Forschungsgemeinschaft" (DFG, German research foundation) within the research scholarship ``Quermassintegral preserving local curvature flows", grant number SCHE 1879/3-1.
\bibliographystyle{amsalpha-nobysame}
\bibliography{Bibliography.bib}

\vspace{10mm}

\textsc{Department of Mathematics, Macquarie University,\\ NSW 2109, Australia, }\email{\href{mailto:paul.bryan@mq.edu.au}{paul.bryan@mq.edu.au}}

\vspace{2mm}

\textsc{Institut f\"{u}r Diskrete Mathematik und Geometrie,\\ Technische Universit\"{a}t Wien,
Wiedner Hauptstr 8-10,\\ 1040 Wien, Austria, }\email{\href{mailto:mohammad.ivaki@tuwien.ac.at}{mohammad.ivaki@tuwien.ac.at}}

\vspace{2mm}

\textsc{Cardiff University, School of Mathematics, Senghennydd Road, Cardiff CF24 4AG, Wales, }\email{\href{mailto:scheuerj@cardiff.ac.uk}{scheuerj@cardiff.ac.uk}}

\end{document}